\documentclass[11pt]{amsart}
\usepackage{amsmath,amsfonts,amssymb,amsthm,amscd,latexsym,euscript,verbatim,
bbold}
\usepackage[all]{xy}

\makeatletter
\def\section{\@startsection{section}{1}%
  \z@{1.1\linespacing\@plus\linespacing}{.8\linespacing}%
  {\normalfont\Large\scshape\centering}}
\makeatother

\theoremstyle{plain}

\newtheorem*{thmAl}{Albert's Fusion Theorem}

\newtheorem*{conj*}{Root Groups Conjecture}
\newtheorem*{thm1.2}{(1.2) Theorem}
\newtheorem*{thm1.3}{(1.3) Theorem}
\newtheorem*{thm1.4}{(1.4) Theorem}
\newtheorem*{prop*}{Proposition}
\newtheorem*{thm*}{Theorem}

\def\eroman{\etype{\roman}}
\def\diag{\operatorname{diag}}

\newtheorem{prop}{Proposition}[section]

\newtheorem{thm}[prop]{Theorem}
\newtheorem{cor}[prop]{Corollary}
\newtheorem{lemma}[prop]{Lemma}

\theoremstyle{definition}

\newtheorem*{Def*}{Definition}
\newtheorem{Defs}[prop]{Definitions}

\newtheorem{examples}[prop]{Examples}

\newtheorem{notation}[prop]{Notation}
\newtheorem*{notation*}{Notation}
\newtheorem{remark}[prop]{Remark}

\newtheorem{rem}[prop]{Remark}

\newcommand{\etype}[1]{\renewcommand{\labelenumi}{(#1{enumi})}}

\newcommand{\la}{\lambda}

\newcommand{\ff}{F}

\newcommand{\zz}{\mathbb{Z}}

\newcommand{\ga}{\alpha}
\newcommand{\gb}{\beta}
\newcommand{\gc}{\gamma}

\newcommand{\gd}{\delta}

\newcommand{\gre}{\epsilon}
\newcommand{\gl}{\lambda}

\newcommand{\gr}{\rho}
\newcommand{\gs}{\sigma}

\newcommand{\gt}{\tau}

\newcommand{\charc}{{\rm char}}

\newcommand{\Aut}{{\rm Aut}}

\newcommand{\half}{\textstyle{\frac{1}{2}}}

\newcommand{\widebar}[1]{\overset{\mskip1mu\hrulefill\mskip1mu}{#1}
                \vphantom{#1}}

\numberwithin{equation}{section}

\begin{document}
\title[Axes in non-associative algebras]{Axes  in non-associative algebras}
\author[Louis Rowen, Yoav Segev]
{Louis Rowen$^*$\qquad Yoav Segev}

\address{Louis Rowen\\
         Department of Mathematics\\
         Bar-Ilan University\\
         Ramat Gan\\
         Israel}
\email{rowen@math.biu.ac.il}
\address{Yoav Segev \\
         Department of Mathematics \\
         Ben-Gurion University \\
         Beer-Sheva 84105 \\
         Israel}
\email{yoavs@math.bgu.ac.il}
\thanks{$^*$The first author was supported by the Israel Science Foundation grant 1623/16}


\begin{abstract}
``Fusion rules'' are laws of multiplication among eigenspaces of an
idempotent. This terminology is relatively new and is closely
related to axial algebras, introduced recently by Hall, Rehren and
Shpectorov. Axial algebras, in turn, are closely related to
$3$-transposition groups and Vertex operator algebras.

In this paper we consider fusion rules for semisimple idempotents,
following Albert in the  power-associative case.  We examine the
notion of an axis in the non-commutative setting and show that the
dimension $d$ of any
 algebra $A$ generated by a pair $a,b$ of (not necessarily Jordan)
axes of respective types $(\gl,\gd)$ and $(\gl',\gd')$ must be at
most $5$; $d$ cannot be $4.$ If $d\le 3$ we list all the
possibilities for $A$ up to isomorphism.

We prove a variety of additional results and mention some research questions at the end.
\end{abstract}

\date{\today}
\maketitle
\section{Introduction}
Our goal in this paper is to further the study of axes, i.e.,
semi-simple idempotents, in an arbitrary algebra $A.$
 Throughout this paper $A$ is an algebra (not necessarily
associative or commutative, not necessarily with a multiplicative unit element)
over a field $\ff.$

Let $a\in A$ be an idempotent, and $\gl,\gd\in\ff.$

\begin{enumerate}\eroman
\item We define the {\it left and right} multiplication maps
$L_a(b) := a\cdot b $ and
$R_a(b) := b\cdot a.$

 \item
We write $A_{\la}(X_a)$ for the eigenspace
 of $\la$ with respect to the transformation $X_a$, $X\in \{L,R\},$ i.e.,
$A_{\la}(L_a) = \{ v \in A: a\cdot v = \la v\},$
and similarly for $A_{\la}(R_a).$ Often we just write $A_\la$ for
$A_{\gl}(L_a)$, when $a$ is understood.
We note that it may happen that $A_{\gl}(X_a)=0.$

\item
We write $A_{\gl,\gd}(a):=A_{\gl}(L_a)\cap A_{\gd}(R_a).$
We just write $A_{\gl,\gd }$ when the idempotent $a$ is  understood.
An element in $A_{\gl,\gd }(a)$ will be called a $(\gl,\gd)$-{\it
eigenvector} of $a,$  and $(\gl,\delta)$ will be called its {\it
eigenvalue}
\end{enumerate}

An idempotent $a\in A$ is a {\it left axis}  (resp.~{\it right} axis) if $L_a$ (resp.~$R_a$) is a semi-simple operator.
An {\it axis} is an idempotent $a$ which is both a left and a right axis,
and satisfies $L_aR_a=R_aL_a.$
If $A$ is associative with an identity element,  then $A=A_{1,1}(a)+A_{0,0}(a)+A_{1,0}(a)+A_{0,1}(a),$
for any idempotent $a\in A.$

Axes $a$ in a power-associative algebra $A$ (although not under that name)
were already studied by Albert \cite{A}.  He showed that
when $A$ is commutative, $A=A_{1,1}(a)+A_{0,0}(a)+A_{\tiny{1/2,\, 1/2}}(a),$
for any idempotent $a\in A.$
In particular this is the case when $A$ is a Jordan algebra.

In \cite{HRS}, Hall, Rehren and Shpectorov introduced
axial algebras.  These are commutative algebras, which
are not necessaily power associative, and which are generated by axes.
Axial algebras, and, in particular,
``primitive axial algebras of Jordan type'' are of interest
because of their connection with
group theory and with Vertex operator algebras.  We refer the reader
to the introduction of \cite{HRS} for further information.

Axial (composition) algebras are
(non-associative) algebras generated by axes.
These algebras are not necessarily power associative
and not necessarily commutative.  However, they
are generated by axes satisfying certain {\it fusion rules}
(see \cite{DPSC} for the most general notion of fusion
rules and for the notion of decomposition algebra).

Our main interest in this paper, continuing our work in \cite{RoSe},
is to place (commutative) axial algebras
  in a general non-commutative
framework.  We hope that in addition to being
interesting in its own right, this information might be used to
understand (and prove) the finite dimensionality of various finitely
generated primitive axial algebras.

\begin{Defs}\label{defs main20}$ $
\begin{enumerate}
\item
The algebra $A$ is {\it flexible} if it
satisfies the identity $(xy)x=x(yx).$
In a flexible algebra we write $xyx$ without
parentheses since there is no ambiguity.

\item
$A$ is {\it power-associative} if $F[x]$ is associative (and therefore
commutative) for each $x \in A.$
\end{enumerate}
\end{Defs}

Commutative algebras are flexible since
\begin{equation}\label{flex1}
(xy)x=(yx)x=x(yx).
\end{equation}

We first study eigenvalues of idempotents $a\in A$ in the spirit of
\cite{A}.  We often assume  flexibility.
In subsection \ref{id} we also often assume that $A$ is power-associative,
relying heavily on  Albert~\cite{A}.

\begin{Defs}\label{defs main3}$ $
\smallskip
Let $a\in A$ be an idempotent, and $\gl,\gd\notin\{0,1\}$ in $\ff$.
\begin{enumerate}
\item
 $a$ is a {\it left axis}  if $a$ is a left semisimple idempotent,
 i.e.,  $L_a $ satisfies a polynomial $p[t]=t(t-1)\prod_{i=1}^m (t - \gl _i)$ with only simple roots $0, 1, \gl_1, \dots, \gl _m.$
 Note that we do not require $p[t]$ to be the minimal polynomial of $L_a.$  We call $\gl_1, \dots, \gl _m$
 the {\it type} of $a.$ The left axis $a$ is {\it primitive} if $A_1 = \ff a$.

For a primitive left axis $a$ of type $\gl_1, \dots, \gl _m,$ and $x\in A,$ we write
\[
\textstyle{x=\ga_x a+x_0+\sum_{i=1}^m x_{\gl_i},\qquad x_{\gl_i}\in A_{\gl_i}(L_a).}
\]
and we call it {\it the left decomposition of $x$ with respect to $a$}.

 \item
Given a left axis $a,$ let $B:=A_1+A_0.$ We say that
$a$ satisfies the {\it left basic fusion rule} if
\begin{itemize}
\item[(i)]
$B$ is a subalgebra of $A.$

\item[(ii)]
$BA_\gl =  A_\gl B\subseteq A_{\gl}$ for each eigenvalue $\gl \ne 0,1.$

\item[(iii)] For each $\gl$ there is $\gl'$ such that $A_{\gl}A_{\gl '} \subseteq B.$
\end{itemize}
The {\it left involutory fusion rules} are the basic
fusion rules together with $A_{\gl}^2\subseteq B,$ for each $\gl.$

\item
Thus $a$ is a primitive left  axis {\it  of type $\gl$} with the
left involutory fusion rule if
\begin{itemize}
\item[(i)]
$(L_a-\gl)(L_a-1)L_a =0.$ (In particular the only left eigenvalues
are contained in $\{0,1,\gl\}$.)

\item[(ii)]  There is a direct sum decomposition of
$A$ which is a $\zz_2$-grading:
\[
A=\overbrace{\ff a\oplus A_{0}}^{\text {$+$}}\oplus
\overbrace{A_{\gl}}^{\text{$-$}},
\]
recalling that $A_{\gl}$ means $A_{\gl}(L_a)$.
\end{itemize}

\item
A {\it right axis of type $\gd_1,\dots,\gd_n$} is defined similarly.
Also the right fusion rules are defined as in (2). As with a left axis,
for a right axis~$a$ and $x\in A,$ we write
\[
\textstyle{x=\gb_x a+\, {}_0x+\sum_{i=1}^m \, {}_{\gl_i}x,\qquad {}_{\gl_i}x\in A_{\gl_i}(R_a).}
\]
and we call it {\it the right decomposition of $x$ with respect to $a$}.

\item
$a$ is a (2-sided) {\it primitive axis   of type $(\gl_1,\dots,
\gl_m; \gd_1,\dots,\gd_n)$ } if $a$ is a primitive left axis of
type $\gl_1,\dots, \gl_m$ and a primitive right axis of type
$\gd_1,\dots,\gd_n,$ satisfying the left and right involutory fusion rules, and, in addition,
 $L_a R_a =  R_a  L_a.$
 In particular, $a$ is a primitive  axis   of type  $(\gl,\gd),$ when $m=n=1.$

Note that for any primitive axis $a$, $A_{1,1} =\ff a =A_1(L_a)=A_1(R_a)$; in
particular $A_{1,\mu} = A_{\nu,1} = 0,$ for any $\nu,\mu$ distinct from $1.$
So if $a$ is a primitive axis, and $x\in A,$ we write
\[
\textstyle{x=\ga_x a+x_{0,0}+\sum_{i=1}^m x_{\gl_i,0}+\sum_{j=1}^n x_{0,\gd_i}+\sum_{i,j}x_{\gl_i,\gd_j}.}
\]
and we call it {\it the decomposition of $x$ with respect to $a$}.

\item
(\cite{RoSe})  An axis $a$ is of {\it Jordan type $(\gl,\gd)$}
if $A_{\gl,0}(a)=0=A_{0,\gd}(a).$  Note that in this case
\[
x=\ga_x a+x_{0,0}+x_{\gl,\gd},\qquad \forall x\in A.
\]
In this case we sometimes call $a$ an axis of Jordan type.

\end{enumerate}
\hfill$\Diamond$
\end{Defs}

Hence, if $a$ is an axis of type $(\gl,\gd),$ then
\[
A=\overbrace{A_{1,1} \oplus A_{0,0}}^{\text {$++$}}\oplus
\overbrace{A_{0,\gd}}^{\text{$+-$}} \oplus \overbrace{
A_{\la,0}}^{\text {$-+$}}\oplus \overbrace{A_{\gl,\gd
}}^{\text{$--$}},
\]
is $\zz_2\times\zz_2$ grading of $A$ (multiplication $(\gre,\gre')(\gr,\gr')$
is defined in the obvious way for $\gre,\gre',\gr,\gr'\in\{+,-\}$).

 If $a$ is an axis of Jordan type $(\gl,\gd),$ then
 \[
A = A_{1} \oplus A_{0} \oplus A_{\gl,\gd }.
\]
Jordan type is close to
commutative, but there are natural non-commutative examples given in
\cite[Examples~2.6]{RoSe}. Also see Examples~\ref{notJJ} for axes
not of Jordan type.

Note that besides being non-commutative,   Definition \ref{defs
main3}(2) {\bf does not} require the condition that $A_{0}$ is a
subalgebra, contrary to the usual hypothesis in the commutative
theory of primitive axial algebras. But this condition does not seem
to pertain to any of the proofs.

If $a$ is a primitive axis of type $(\gl,\gd),$ then by
Definition~\ref{defs main3}(5),  we can write $x\in A$ as
\begin{equation}\label{eq decomposition}
x=\ga_x a+x_{0,0}+x_{0,\gd}+x_{\gl,0}+x_{\gl,\gd},\qquad
x_{\mu,\nu}\in A_{\mu,\nu}(a),
\end{equation}
which we call  the {\it $(2$-sided$)$ decomposition of} $x$ with respect
to $a.$

When $a$ is of Jordan type, we have the much simpler decomposition
\[
x = \ga_xa +x _0 + x_{\gl,\gd}
\]
with $x_0\in A_{0,0},$ and $x_{\gl,\gd}\in A_{\gl,\gd },$ which is both
the left decomposition and right decomposition.

Next, for $\gl\in\ff,$ write
\begin{equation}\label{eigenspaces}
\mathring{A}_{\gl}(a) =\{x\in A\mid ax+xa=2\gl x\}.
\end{equation}
(We write $ \mathring{A}_{\gl},$ when $a$ is understood.) Albert
proved for $A$ power-associative  with $\charc(F)\ne 2,3,$ that
\[
A=A_{1,1} \oplus A_{0,0}\oplus \mathring{A}_{1/2}.
\]
This is reproved   directly as Theorem~\ref{axi}(i). The following
theorem of Albert then describes the appropriate fusion laws in a
power-associative algebra~$A:$

%
\begin{thmAl}[{\cite[Theorem 3, p.~560, Theorem 5, p.~562]{A}}]
Suppose that $A$ is power-associative with $\charc(F)\ne 2,3.$ Let
$a\in A$ be an idempotent (not necessarily primitive).

\begin{enumerate}\eroman
\item
$A=A_{1,1} \oplus A_{0,0}\oplus \mathring{A}_{1/2}.$

\item
$A_{0,0}, A_{1,1}$ are subalgebras.

\medskip
\noindent Furthermore for $A$ flexible,
\item  $A_{\la,\gl} \mathring{A}_{1/2} \subseteq A_{1 -\la,1 -\la}+ \mathring{A}_{1/2}, \quad \mathring{A}_{1/2} A_{\la,\la}  \subseteq A_{1 -\la,1 -\la}+ \mathring{A}_{1/2} ,\quad$ for
$\la = 0,1.$
\item
$a \mathring{A}_{1/2} ,  \mathring{A}_{1/2} a\subseteq
\mathring{A}_{1/2}.$\hfill {$\Diamond$}
\end{enumerate}
\end{thmAl}

We prove
\smallskip

\noindent
{\bf Theorem A} (Theorem \ref{thm 2-gen}).
Assume that $A$ is generated by two axes $a, b$
of type $(\gl,\gd)$ and $(\gl',\gd').$
Then the dimension of $A$ is at most $5.$
\medskip

In \cite[Theorem 2.5]{RoSe}, we classified all algebras
as in Theorem A, of dimension $2.$  We prove
\smallskip

\noindent
{\bf Theorem B} (Theorem \ref{thm dim3}).
Assume that $A$ is generated by two axes $a, b$
of type $(\gl,\gd)$ and $(\gl',\gd').$ If $\dim(A)=3,$
then $a$ and $b$ are of Jordan type, and hence
the possibilities for $A$ are given in \cite[Theorem B]{RoSe}.
\smallskip

\noindent
{\bf Proposition C}.
Assume that $A$ is generated by two axes $a, b$
of type $(\gl,\gd)$ and $(\gl',\gd').$  Then
\begin{enumerate}
\item
(Lemma \ref{ab in Fa+Fb}) If $ab\in\ff a+\ff b,$ then $\dim(A)=2.$

\item
(Lemma \ref{lem ab=0}) If $ab=0,$ then $ba=0.$
\end{enumerate}

We prove a number of additional results and we mention
some questions in \S5.

\begin{remark}
Along the period that this paper was refereed we
proved that there are no algebras of dimension $4$ generated by
two primitive axes of type $(\gl,\gd)$ and $(\gl',\gd').$  This
result will appear somewhere else.  Thus, in view of Theorem A, it remains to consider
such algebras of dimension $5.$  See the questions in \S 5.
\end{remark}

\section{General results about axes}\label{nc11}$ $

The next proposition  provides a useful decomposition result into
eigenspaces.

\begin{prop}\label{prop decomposition}
Suppose $a$ is a left axis, where $L_a$ satisfies the polynomial $\prod_{i=1}^m(x-\mu_i),$
with only simple roots $\mu_1, \dots, \mu_m.$
For $y \in A$, write $y = \sum_{j=1}^m  y_j,$ where $y_j$ is the
left $\mu_j$-eigenvector in the $L_a$-eigenvector decomposition of
$y$, and define the vector space $V_a(y) = \sum _{i=0}^{m-1}
\ff\L_a^iy.$ Then
\[V_a(y) =
\oplus_{j=1}^m \ff y_j.
\]
\end{prop}
\begin{proof}
By induction, $L_a^k (y) = \sum \mu_j^k y_j$ for $0 \le k \le m-1.$
Thus $$V_a(y) \subseteq \oplus_{j=1}^m \ff y_j.$$ But on the other
hand, we have a system of $m$ equations in the $y_j$ with
coefficients $( \mu_j^k)$, the Vandermonde matrix, which is
non-singular, so there is a unique solution for the $y_j$ in
$V_a(y)$.
\end{proof}

This space $V_a(b)$ plays a key  role in investigating a second axis
$b$ (see Proposition \ref{prop V_a(b)}).

\subsection{Basic properties of flexible idempotents}\label{id}$ $
\smallskip

\noindent
An idempotent $a\in A$  is called {\it flexible} if $(ax)a = a(xa)$
and $(xa)x = x(ax)$ for all $x \in A$. We present some examples of
axes that are not flexible, to justify full generality.

\begin{examples}\label{notJJ}$ $
\begin{enumerate} \eroman
\item
Let $A = \ff a\oplus \ff b\oplus \ff c$ with multiplication given by
\[
 \begin{aligned}
&a^2 = a,  \qquad x^2 = a = y^2,\\ & ax = \gl x, \qquad xa = 0,\\
& ay = 0, \qquad ya = \gd y, \qquad xy =   yx = 0.
\end{aligned}
 \]
 Then
 $a(xa) = 0 = (ax)a$, and   $a(ya) =  0= (ay)a,$ so $L_a R_a = R_a L_a.$
\[
(L_a - \gl )x =\gl(x -x) =0 = L_a y,
\]
implying $(L_a-\gl)(L_a-1)L_a =0.$
Likewise $(R_a-\gd)(R_a-1)R_a =0.$

Hence $a$  is a primitive axis of  type $(\gl,\gd)$ which is not
of Jordan type, with $\ff x = A_{\gl,0}(a)$ and $\ff y = A_{0,\gd}(a)$; $a$ is the only
axis in $A$.

\item
Let $0 \ne \gl,\gd\in \ff,$ with $\gl+\gd = 1.$ Let
\[
A=Fa\oplus Fc\oplus Fx\oplus Fy\oplus Fz,
\]
with multiplication defined by:
\begin{gather*}
a^2=a,\quad ac=ca=0,\quad ax=\gl x, xa=0,\quad ay=0, ya=\gd y,\\
az=\gl z,\quad za=\gd z.
\end{gather*}
\begin{gather*}
c^2=cx=xc=cy=yc= cz=zc=0.\\
x^2=-c,\quad xy=0=yx,\quad xz=\gl y,\quad zx=0.\\
y^2=c,\quad yz=\gd x,\quad zy=0,\quad z^2=0.
\end{gather*}
So
\[
A_{1,1}=Fa,\quad A_{0,0}=Fc,\quad A_{\gl,0}=Fx,\quad
A_{0,\gd}=Fy,\quad A_{\gl,\gd}=Fz.
\]
It is easy to check that $(av)a=a(va),$ for $v\in\{a,c,x,y,z\},$
and $a$ is a primitive axis of type $(\gl,\gd),$  not of Jordan
type. Also, $a$ is not a flexible axis since $(xa)x=0,$ while
$x(ax)=\gl x^2=-\gl c.$

Let $b = a + c+x+y+z.$ Then
\[
\begin{aligned}  b^2 = & a^2 +x^2  +(y)^2 +ax +ya + az +za +xz + yz  \\
&= a - c+c +(\gl + \gd) x  + (\gl + \gd) y + (\gl + \gd)z= b,
\end{aligned}
\]
Note that $A$ has dimension  5. Since
$a,b,ab,ba,$ and $aba$ are independent, they span $A$.
\end{enumerate}
\end{examples}

Flexible idempotents are rather special.

\begin{lemma}\label{deg7}
If $a$ is an idempotent satisfying $(xa)x=x(ax),$ for all $x\in A,$ then $L_a ^2 -L_a =  R_a ^2 -R_a.$
\end{lemma}
\begin{proof}
We have:
\begin{gather*}
a +ax + (xa)a + xax =  (a+xa)(a+x) = ((a+x)a)(a+x) =\\
(a+x)(a(a+x)) =  (a+x)(a+ax) =a+a(ax) + xa + xax,
\end{gather*}
so $ax + (xa)a =a(ax) + xa$.
\end{proof}

\begin{lemma}\label{acationscom20}
Suppose  $a\in A$ is an idempotent satisfying
 $ L_a^2 -    L_a =R_a^2 - R_a  .$  (In particular this holds when $a$ is an
idempotent satisfying $(xa)x=x(ax),$ for all $x\in A,$  by
 Lemma~\ref{deg7}.)
\begin{enumerate}
\item
If $A_{\mu,\nu}(a)\ne 0,$
then either $\nu =\mu$    or $\nu =1-\mu.$

\item
$A$ decomposes as $\oplus _{\gl} A_{\gl,\gl} \oplus A_{\gl,1-\gl},$
summed over all left eigenvalues $\gl$ of~$a$.
\end{enumerate}
\end{lemma}
\begin{proof}
(1)\ Let $z\ne 0$ is a  $(\mu,\nu)$-eigenvector, then
\[
(\mu^2 - \mu)z = a(az) - az = (za)a-za= (\nu^2 -\nu)z,
\]
implying $\mu^2 - \mu=\nu^2 - \nu,$ or $(\mu -
\nu)(\mu + \nu - 1) = 0.$
\medskip

\noindent
(2)\ These are the only possibilities, in view of (1).
\end{proof}

\begin{lemma}\label{acationscom}
Suppose  $a\in A$ is a  flexible idempotent. Then
\[
R_a(R_a+L_a-1)=L_a(R_a+L_a-1).
\]
\end{lemma}
\begin{proof}
 By Lemma~\ref{deg7}, and since $L_aR_a=R_aL_a,$
\begin{gather*}
R_a(R_a+L_a-1)=R_aL_a+R_a^2-R_a=R_aL_a+L_a^2-L_a\\
=L_a(R_a+L_a-1).\qedhere
\end{gather*}
\end{proof}

\subsection{Albert's power-associative results}$ $
\smallskip

\noindent
As an application, we reprove Albert's theorem on power-associative
algebras, but first we need:

\begin{lemma}\label{acationscomii}
If  $A$ is flexible and power-associative over a field of characteristic $\ne
2,3,$ then
\[
(X_a - 1)  Y_a (L_a + R_a -1) = 0,\text{ for }X,Y\in \{R,L\}.
\]
 \end{lemma}
\begin{proof} We show that $(R_a - 1)  R_a (L_a + R_a -1) = 0,$ the rest
follows from Lemma \ref{deg7} and Lemma \ref{acationscom}.

By \cite[Equation 12, p. 556]{A}, and applying Lemma~\ref{deg7},
 \begin{equation}\begin{aligned}  0 & = R_a^2 + L_a R_a+ L_a^2 - R_a^3 -L_a R_a^2 -L_a \\
 & = R_a^2  + L_a R_a + R_a^2 - R_a^3 -L_a R_a^2 -R_a  \\ &= 2
R_a^2 - R_a^3 +L_a (R_a -R_a^2) -R_a ,\end{aligned}\end{equation}

so
\[(R_a - 1)^2 R_a =  R_a^3 - 2 R_a^2 +R_a  = L_a (R_a -R_a^2) =
- L_a R_a  (R_a -1) ,
\]
so
\[
(R_a - 1)  R_a (L_a + R_a -1) = 0.\qedhere
\]
\end{proof}

\begin{remark}\label{not06}
In  an easy argument in the first three lines of the proof of
\cite[Theorem 3, p.~560]{A} Albert proves  that
$\mathring{A}_{\gl}=A_{\gl,\gl}$ for $\gl\in\{0,1\}$ (see Equation
\eqref{eigenspaces} for notation). This is used
 in the following result  of Albert on power-associative
algebras (\cite[(22), p.~559]{A} and \cite[Thm.~3, p.~562]{A}).
\end{remark}

\begin{thm}[\cite{A}]\label{axi}
Suppose that $A$  is power-associative, and  that $\charc(\ff)\ne
2,3.$ Let $a\in A$ be an idempotent. Then
\begin{enumerate}\eroman
\item
$A=A_{1,1}\oplus A_{0,0}\oplus \mathring{A}_{1/2}.$

\item
We have
\begin{gather*}
A_1 = X_a (L_a+R_a -1)A; \qquad A_0 = (X_a-1) (L_a+R_a-1)A; \\
  \mathring{A}_{1/2} = (R_a+L_a)(R_a+L_a-2)A, \quad  X\in\{R,L\}.
    \end{gather*}
\end{enumerate}
\end{thm}
Here is an alternate proof that is rather conceptual and reasonably
quick, working directly in~$A.$
\begin{proof}
(i)\ Assume first that $A$ is commutative.  Then $A$ is flexible
(see Equation \eqref{flex1}), so we may use Lemma~\ref{acationscom}
and Lemma \ref{acationscomii}.  Let $x\in A.$

Let $x_1=R_ax$ and $x_2=(1-R_a)x,$ so that $x=x_1+x_2.$ By
Lemma~\ref{acationscomii}  we have
\begin{equation}\label{eq axi(1)}
(R_a-1)(R_a+L_a-1)(x_1)=0=(L_a-1)(R_a+L_a-1)(x_1).
\end{equation}
and
\begin{equation}\label{eq axi(2)}
R_a(R_a+L_a-1)(x_2)=0=L_a(R_a+L_a-1)(x_2).
\end{equation}
Let
\begin{alignat*}{3}
x_1'&=(R_a+L_a-1)(x_1), &\qquad\qquad & x_1''=(R_a+L_a-2)(x_1),\\
x_2'&=(R_a+L_a-1)(x_2), &\qquad\qquad & x_2''=(R_a+L_a)(x_2),
\end{alignat*}
Of course
\[
x=x_1'-x_1''+x_2''-x_2'.
\]
By Equations \eqref{eq axi(1)} and \eqref{eq axi(2)},
\[
R_a(x_1')=x_1'=L_a(x_1'),\qquad R_a(x_2')=0=L_a(x_2').
\]
So
\[
x_1'\in A_1,\qquad\text{and}\qquad x_2'\in A_0.
\]
Also by Equation \eqref{eq axi(1)}, $(R_a+L_a-1)(x_1'')=0,$ that is
$ x_1''\in \mathring{A}_{1/2}, $ and by Equation \eqref{eq axi(2)},
$(R_a+L_a-1)(x_2'')=0,$ that is $ x_2''\in \mathring{A}_{1/2}.$

For the general case where $A$ is not commutative, $\mathring{A}$ is
commutative and power-associative; hence
$A=\mathring{A}_1\oplus\mathring{A}_0\oplus \mathring{A}_{1/2}.$
However, $\mathring{A}_{\gl}=A_{\gl,\gl},$ for $\gl\in\{0,1\},$ by
Remark~\ref{not06}.
\medskip

\noindent (ii)\ follows from (i). Indeed
\[
X_a (L_a+R_a -1)A_0=0=X_a (L_a+R_a -1)\mathring{A}_{1/2},
\]
and $X_a (L_a+R_a -1)(x)=x,$ for $x\in A_1,$ and similarly for
$(X_a-1) (L_a+R_a -1)$ and $(X_a+L_a)(R_a+L_a-2).$
\end{proof}

Note that if $A$ is commutative then $\mathring{A}_{1/2} =A_{1/2} ,$
so Theorem~\ref{axi} generalizes the commutative description of
axes.

%
\subsection{Properties of primitive axes of type $(\gl,\gd)$}$ $
\smallskip

 Here are some computations.

\begin{lemma}\label{lem some computations}$ $
Suppose  $a$ is a primitive axis of type $(\gl,\gd),$ and let $y\in
A.$
\begin{enumerate}
\item
$a(\ff a+A_0)= \ff a .$

\item
$ay=\ga_y  a+\gl y_{\gl},$ so
$y_{\gl} =\frac 1 {\gl}(ay -\ga_y  a).$

\item
$a(ay)=\ga_y (1-\gl) a+\gl ay.$

\item $ y_0 =  y -\frac 1 \gl
 (ay -(\gl +1)\ga_y  a)  .$

 \item
$(ay)a =\ga_y  a+\gl \gd y_{\gl,\gd},$ implying
$y_{\gl,\gd}=\frac 1{\gl \gd}((ay)a -\ga_y  a).$

\end{enumerate}
\end{lemma}
\begin{proof}
(1) This is obvious.

(2) Apply $L_a$ and solve.

(3) $ a(ay)=\ga_y  a+\gl^2 y_{\gl}=\ga_y  a+\gl \textstyle{(ay-\ga_y
a)},$  yielding (3).

 (4) $y_0 = y - y_\gl - \ga_y  a = y -\frac
1 \gl
 (ay -(\gl +1)\ga_y  a).$

 (5) Apply $R_a$ to (2).
\end{proof}

\begin{lemma}\label{lem 1.660}$ $
\begin{enumerate}\eroman
\item $ab \notin \ff a$ or $ab =0$, for any   idempotent $a$ and  primitive right axis $b$
satisfying the involutory fusion rules.

\item For any primitive left axis $a$ of type
$\gl$ and any primitive right axis~$b$ satisfying the involutory
fusion rules, either $ab =0$ or $b_{\gl}\ne 0$.

\item If $a$ is a primitive axis of Jordan type $({\gl,\gd})$ with $\gl\ne \gd,$
and $b$ is an  axis, then $b_{\gl,\gd}^2 = 0.$
\end{enumerate}
\end{lemma}
\begin{proof}
(i) If $ab = \gc a,$ then $a$ is an eigenvector of $R_b,$ with
eigenvalue $\gc.$  If $\gc=1,$ then $b$ is not primitive, a
contradiction.  Otherwise, by the fusion rules for $b,$ $a=a^2\in\ff
b+A_0(R_b).$  But then $ab\in\ff b.$ This together with $ab\in\ff a$
forces $ab=0.$

(ii) If $b_{\gl}= 0,$ then $ab = \ga_b a,$ so we are done by (i).

(iii)   The proof is exactly as in \cite[Lemma 2.12(iii)]{RoSe}.
\end{proof}

\begin{cor}\label{Jt}$ $

\begin{enumerate}\eroman
\item A primitive axis  $a$ of type  $(\gl,\gd)$ is in the center of $A$, iff it is of Jordan type
with $A_{\gl,\gd }=0.$

\item
Let $a\ne b$ be two primitive axes  with $ab=ba$ and $a$ of type  $(\gl,\gd)$.
Then either $ab=0,$ or $\gl = \gd.$ In particular, if $a$ is of
Jordan type, then either $a$ commutes with all elements of $A$ or $ab=0$.
\end{enumerate}
\end{cor}
\begin{proof}
(i) $(\Rightarrow)$ For any $x\in A$,
\[
\ga_x a+\gl^2 x_\gl  = a(ax) = ax =\ga_x a+ \gl x_\gl,
\]
implying $x_\gl = 0$ since $\gl \ne 0,1.$ Hence $x_{\gl,0}=0 =
x_{\gl,\gd},$ and by symmetry $x_{0,\gd}=0.$

$(\Leftarrow)$ $x = \ga _x a+x_{0,0}$ for any $x$ in $A$, implying
$ax = xa = \ga_x a$ and also $xy = x_{0,0}y_{0,0} + \ga _x \ga _y
a,$ and thus $a(xy) = (ax)y = \ga _x a y = \ga _x \ga_y a$ for all
$x,y$.%

(ii) Let $b=\ga_b a+b_{0,0}
+b_{\gl,0} +b_{0,\gd}+b_{\gl,\gd}$ be the decomposition of  $b$ with
respect to~$a$. Then $\ga_b a +\gl b_{\gl,0}+\gl b_{\gl,\gd}
=ab=ba=\ga_b a +\gd b_{0,\gd}+\gd b_{\gl,\gd}$ implying
$b_{\gl,0}=b_{0,\gd}=0.$ If $b_{\gl,\gd}\ne 0,$ then $\gl =\gd$, and
we are done. Hence $ b_{\gl,\gd}=0$, in which case $ab \in \ff a,$ so
by Lemma \ref{lem 1.660}(i), $ab =0.$
\end{proof}
\smallskip

\noindent
We generalize Seress' Lemma from  \cite[Lemma 2.7]{RoSe}:

\begin{lemma}[General Seress' Lemma]\label{lem seress}
If $a$ is a primitive left axis  satisfying the left basic fusion rules
then $a(xy)=(ax)y+a(x_0y),$   for any $x \in A$ and $y \in \ff a +A_0(L_a).$
In particular, $a(xy) \in (ax)y + \ff a.$ Here
\[
x=\ga_xa+x_0+\sum_{i=1}^mx_{\gl_i},
\]
is the left decomposition of $x$ with respect to $a.$

The symmetric statement holds:
If $a$ is a primitive right axis  satisfying the right basic fusion rules
then $(yx)a=y(xa)+a(y\cdot {}_0x),$   for any $x \in A$ and $y \in \ff a +A_0(R_a).$
In particular, $(yx)a \in y(xa) + \ff a.$
\end{lemma}
\begin{proof}
By linearity we need only check for $x$ an eigenvector of $L_a.$
If $x = x_0 \in A_0(L_a),$ then $a(x_0y)=0+a(x_0y)=(ax_0)y+a(x_0y).$
By the basic fusion rules $x_0y\in \ff a+A_0(L_a),$ so $a(x_0y)\in\ff a.$

If $x=x_{\mu} \in A_{\mu}(L_a)$ for $\mu \ne \{0,1\}$ then $x_{\mu}y\in A_{\mu}(L_a),$
by the basic fusion rules, hence $(ax_{\mu})y=a(x_{\mu}y)=\mu x_{\mu}y.$
\end{proof}

\section {Miyamoto involutions}

It is easy to check
that any $\mathbb Z_2$-grading of $A$ induces an
{\it involution}, i.e.~an
automorphism of order $2$ of $A.$  Indeed, if $A=A^+\oplus A^-,$
then $y\mapsto y^+-y^-$ is such an automorphism, where $y\in A$ and
$y=y^+ +y^-$ for $ y^+\in A^+, y^-\in A^-.$

So if $a\in A$ is a primitive axis of type $(\gl,\gd),$ then  we
have three such automorphisms of order 2, which, conforming with the literature, we call the {\bf
Miyamoto involutions} {\it associated with $a$}:

\begin{enumerate}\eroman
\item
$\gt_{\gl,a}=\gt_{\gl}: y\mapsto y-2y_{\gl}=\ga_y a+y_{0,0}+y_{0,\gd}-y_{\gl,0}-y_{\gl,\gd}.$

\item
$\gt_{\gd,a}=\gt_{\gd}: y\mapsto y-2\, {}_{\gd}y=\ga_y a+y_{0,0}-y_{0,\gd}+y_{\gl,0}-y_{\gl,\gd}.$

\item
$\gt_{\diag,a}=\gt_{\diag}: y\mapsto \ga_y a+y_{0,0}-y_{0,\gd}-y_{\gl,0}+y_{\gl,\gd}.$
\end{enumerate}
In case $a$ is of Jordan type, the first two Miyamoto involutions
are the same, and the third is the identity, so we are left with a
single non-trivial Miyamoto involution associated to~$a,$ which
we write as $\tau_a.$
\medskip

\begin{notation}\label{not G(X)}
For any Miyamoto involution $\tau$ we write $y^\gt$ for $\gt(y).$
Let $X$ be a set of axes, where each $x\in X$ is primitive of type $(\gl_x,\gd_x).$
We denote by  $G(X)$ the subgroup of $\Aut (A)$ generated by all the
Miyamoto involutions associated with all the axes in~$X.$
We denote $\widebar{X}:=X^{G(X)}=\{x^g\mid x\in X,\ g\in G(X)\}.$
\end{notation}

\begin{remark}\label{Miyfhol}
Let $a$ be a primitive axis of type $(\gl,\gd).$ Of course we can recover:

\begin{enumerate}\eroman
\item
$y_{\gl,\gd}+ y_{\gl,0}= \half (y-y^{\gt_\gl}).$

\item
$y_{\gl,\gd}+ y_{0,\gd}= \half (y-y^{\gt_\gd}).$

\item
$y_{0,\gd}+ y_{\gl,0}=\half(y-y^{\gt_{\diag}}),$
\end{enumerate}
so each of  $ y_{\gl,\gd}, y_{0,\gd},$ and $ y_{\gl,0}$ are linear
combinations of    $y$, $y^{\gt_{\gl}}$, $y^{\gt_{\gd}},$ and
$y^{\gt_{\diag}}$, as is $ +\ga _y a= y - ( y_{\gl,\gd}+
y_{0,\gd}+ y_{\gl,0}).$\hfill$\Diamond$
\end{remark}

\begin{thm}\label{thm KMt}
Suppose $A$ is generated by a set of axes $X,$
where each $x\in X$ is primitive of type $(\gl_x,\gd_x).$
\begin{enumerate}\eroman
\item
(Generalizing \cite[Corollary 1.2, p.~81]{HRS}.)
$A$ is spanned by the set $\widebar{X}$.

\item
Let $V\subseteq A$ be
a subspace containing $X$ such that $xV\subseteq V$ and $Vx\subseteq
V,$ for all $x\in X.$ Then $V=A.$
\end{enumerate}
\end{thm}

\begin{proof}
(i) By induction on the length of a monomial in the axes $\widebar{X},$
it suffices to show that  $ab$ is in the span of $\widebar{X},$ for
$a,b\in \widebar{X}.$ Write
\[
b=\ga_b a+b_{0,0}+b_{\gl,0}+b_{0,\gl}+b_{\gl,\gd}.
\]
Then $ab=\ga_b a+\gl b_{\gl,0}+\gl b_{\gl,\gd}.$  But by Remark
\ref{Miyfhol}, $b_{\gl,0}$  and  $b_{\gl,\gd}$ are linear
combinations of $b, b^{\gt_{\gl}}, b^{\gt_{\gd}}, b^{\gt_{\diag}}.$

(ii) Let $a\in X,$ so that  $aV\subseteq V$ and $Va\subseteq V.$ We claim
that if $v\in V,$ then  $v_{\mu,\nu}\in V$ for each $\mu\in\{\gl,0\},\nu\in\{\gd,0\}.$
First, $v_{\gl,\gd}\in V$ since $\ga _v a +\gl\gd v_{\gl,\gd} =ava
=a(va) \in V.$ Hence $v_{\gl,0}\in V$ since $av - ava\in V,$ and
$v_{0,\gd}\in V$ since $va -ava \in V.$ Finally $v_{0,0} = v -\ga_va
-v_{0,\gd} -v_{\gl,0}- v_{\gl,\gd}\in V.$

Next we claim that $V^{\gt}=V,$ for all $\gt\in\{\gt_{\gl,a},
\gt_{\gd,a}, \gt_{\diag,a}\},$ and all $a\in X.$  Indeed, for $v\in
V,$ $v^{\gt_{\gl}}=v-2(v_{\gl,0}+v_{\gl,\gd}),$
$v^{\gt_{\gd}}=v-2(v_{0,\gl}+v_{\gl,\gd}),$ and
$v^{\gt_{\diag}}=v-2(v_{\gl,0}+v_{0,\gd}),$  Since all $v_{\mu,\nu}
\in V,$ we see that our claim holds.

Hence, in the notation of (i),  $V$ contains $\widebar{X}$, whose span
is $A$.\qedhere
\end{proof}

\section{Algebras generated by $2$ primitive axes
of type $(\gl,\gd)$ and $(\gl',\gd')$}

In this section,  $A$ is  an algebra generated by
$2$ primitive axes $a$ of type $(\gl,\gd)$ and $b$ of type
$(\gl',\gd')$.
When $a$ and $b$ have Jordan type, we showed in \cite{RoSe}
that $\dim(A)\le 3,$ and, furthermore,
we classified the possible algebras $A$ (see \cite[Theorem B]{RoSe}).
In \cite[Theorem A]{RoSe} we classified all algebras $A$ with $\dim(A)=2.$
Here we continue the classification of the possible algebras $A.$
Throughout this section we let
\[
\gs:=ab-\gd'a-\gl b\qquad\text{and}\qquad\gs':=ba-\gl' a-\gd b.
\]
We write
\[
b=\ga_b a+b_{0,0}+b_{0,\gd}+b_{\gl,0}+b_{\gl,\gd}\quad\text{and}\quad a=\ga_a b+a_{0,0}+a_{0,\gd'}+a_{\gl',0}+a_{\gl',\gd'}
\]
For the decomposition of $b$ (respectively $a$) with respect to $a$
(resp.~$b$), as in Equation \eqref{eq decomposition}.

\begin{lemma}\label{lem gs}
\begin{enumerate}\eroman
\item
We have
\[
\gs =\gc a-\gl (b_{0,\gd}+b_{0,0})=\gc' b-\gd'(a_{0,0}+a_{\gl',0}).
\]
where $\gc=\ga_b(1-\gl)-\gd',$ and $\gc'=\ga_a(1-\gd')-\gl.$ So
\[
\gs\in (\ff a+A_0(L_a)) \cap (\ff b+A_0(R_b)),
\]

\item
$a \gs = \gc a$ and $\gs b = \gc' b.$

\item
$\gs' a \in \ff a$ and $b \gs'  \in\ff b.$
\end{enumerate}
\end{lemma}
\begin{proof}
(i)
We compute that
\[
\begin{aligned}
\gs =&(\ga_ba+\gl (b_{\gl,0}+b_{\gl,\gd}))-\gd' a-\gl\ga_ba
-\gl(b_{0,0}+b_{0,\gd})-\gl (b_{\gl,0}+b_{\gl,\gd})\\
&=(\ga_b-\gl\ga_b-\gd')a-\gl (b_{0,0}+b_{0,\gd}).
\end{aligned}
\]
and
\[
\begin{aligned}
\gs&=\ga_a b+\gd'(a_{0,\gd'}+a_{\gl',\gd'})-\gd'\big(\ga_ab+(a_{0,0}+a_{\gl',0})+(a_{0,\gd'}+a_{\gl',\gd'})\big)-\gl b\\
&=(\ga_a-\gd'\ga_a-\gl)b-\gd'(a_{0,0}+a_{\gl',0}).
\end{aligned}
\]

 \medskip

\noindent
(ii) This is obvious.
\medskip

\noindent
(iii) Reverse $a$ and $b$ in (ii).
\end{proof}

\begin{prop}\label{prop 2-gen}$ $
\begin{enumerate}\eroman

\item  $a(ax)  \in \ff a  + \ff ax,$ for all $x\in A.$

\item $(xb)b\in \ff b + \ff xb,$ for all $x\in A.$

\item
$(ab)(ab) - a(bab) \in V := \ff a + \ff b + \ff ab + \ff ba.$

\item
$a(bab), (bab)a, a(bab)a\in V':=\ff a + \ff b + \ff ab + \ff ba+\ff aba.$

\item
Let $V'$ be as in (iv).  Then $aV'\subseteq V'\supseteq V'a.$

\item
$bab\in V',$ where $V'$ is as in (iv).
\end{enumerate}
 \end{prop}
\begin{proof}
 (i)\&(ii) Are special cases of Lemma~\ref{lem some computations}.
\medskip

\noindent
(iii)
Write $x\cong y$ if $x-y\in V.$ We claim that
\begin{align}\label{eq1}
&a(b\gs)\cong a(bab)-\gd'aba\\\label{eq2}
&(ab)\gs\cong (ab)(ab)-\gd'aba
\end{align}
For \eqref{eq1}, we have $b\gs=b(ab-\gd'a-\gl b),$ so
\[
a(b\gs) = a(bab)-\gd'aba -\gl ab \cong a(bab)-\gd'aba.
\]
For \eqref{eq2} we have
\[
(ab)\gs = (ab)(ab-\gd'a-\gl b) = (ab)(ab)- \gd' aba -\gl (ab)b \cong (ab)(ab) - \gd'aba.
\]
Now $\gs \in \ff a + A_0(L_a),$ so
\begin{alignat*}{2}
a(bab) &\cong a(b\gs) +  \gd' aba &\quad\text{(by Equation \eqref{eq1})}\\
&\cong (ab)\gs + \gd' aba &\quad\text{(by Seress' Lemma)}\\
&\cong (ab)(ab) &\quad\text{(by Equation \eqref{eq2}).}
\end{alignat*}
\medskip

\noindent
(iv)
By (iii) and (i),
\[
a\big((ab)(ab)-a(bab)\big)\in V'.
\]
 By  Lemma~\ref{lem some computations}(3),
$a(a(bab))-\gl a(bab)\in V'$
implying
\[
a((ab)(ab))-\gl a(bab)\in V'.
\]

Now $aba=\ga_ba+\gl\gd b_{\gl,\gd},$ so
$b_{\gl,\gd}\in V'.$ Hence $b_{\gl,0}\in V',$ because $ ab=\ga_b
a+\gl b_{\gl,0}+\gl b_{\gl,\gd}\in V'.$ By the fusion rules,
$a((ab)(ab))$ is an $\ff$-linear combination of $a, b_{\gl,0}$ and
$b_{\gl,\gd}.$ Hence $a((ab)(ab))\in V',$ so $a(bab)\in V'.$

By symmetry, $(bab)a\in V',$ and then clearly $a(bab)a\in V'.$
\medskip

\noindent
(v)  This follows easily from (i), (ii) and (iv).
\medskip

\noindent
(vi)
By (iv) and Lemma \ref{lem some computations}(5), $(bab)_{\gl,\gd}\in V'.$
By (iv) and Lemma \ref{lem some computations}(2) and its right counterpart,
\[
(bab)_{\gl,0}, (bab)_{0,\gd}\in V'.
\]

We have
\begin{align*}
bab&=(\ga_ba+b_0+b_{\gl})(\ga_ba+\gl b_{\gl})\\
&=\ga_b^2a+\ga_b\gl^2 b_{\gl}+\ga_bb_0a+\gl b_0b_{\gl}+\ga_bb_{\gl}a+\gl b_{\gl}^2\\
&=\overbrace{\ga_b^2a+\gl b_{\gl}^2}+
\overbrace{\ga_b\gl^2 b_{\gl}+\ga_bb_0a+\gl b_0b_{\gl}+\ga_bb_{\gl}a}^{\in A_{0,\gd}+A_{\gl,0}+A_{\gl,\gd}}.
\end{align*}
Hence $(bab)_{0,0}=\gl(b_{\gl}^2)_{0,0}.$  Now By (iii) and (iv),
$(ab)^2\in V',$ so $b_{\gl}^2\in V',$ by Lemma~\ref{lem some
computations}(2) (with $b$ in place of $y$).  Denoting
$x:=b_{\gl}^2,$ we can write $x=\ga_x a+x_{0,0}+x_{0,\gd}.$ Since
$V'a\subseteq V',$ we get $xa\in V',$ so $\ga_x a+\gd x_{0,\gd}\in
V'.$  It follows that $x_{0,\gd}\in V',$ and consequently
$(bab)_{0,0}=\gl x_{0,0}\in V'.$ Hence we see that
$(bab)_{\mu,\nu}\in V',$ for all $\mu,\nu\in\{0,\gl,\gd\},$ so
$bab\in V'.$
\end{proof}

\begin{thm}\label{thm 2-gen}
$A = \ff a + \ff b + \ff ab + \ff ba +\ff aba,$ and thus
has dimension $\le 5.$
\end{thm}
\begin{proof}
Let
\[
V':=\ff a + \ff b + \ff ab + \ff ba +\ff aba.
\]
By Proposition \ref{prop 2-gen}, $V' = \ff a + \ff b + \ff ab + \ff
ba +\ff aba +\ff bab.$ By Theorem~\ref{thm KMt} it suffices to show
that $aV', bV', V'a, V'b \subseteq V'.$ By symmetry (with respect
both to $a,b$ and to working on the left or on the right), it is
enough to show that $aV'\subseteq V',$ but this is immediate from
Proposition \ref{prop 2-gen}(i).
\end{proof}

\begin{lemma}\label{ba in V}
Let $V = \ff a + \ff b + \ff ab .$
If $ba \in V$, then $V=A.$
\end{lemma}
\begin{proof}
By Proposition~\ref{prop 2-gen}(i), $V$ is closed under $L_a.$ In
particular $aba\in V,$ so $V$ is closed under $R_a.$ Similarly $V$
is closed under $R_b$ by Proposition~\ref{prop 2-gen}(ii), so
$bab\in V,$ and then $V$ is closed under $L_b.$ The Lemma follows
from Theorem~\ref{thm KMt}(ii).
\end{proof}

Next we note that
\begin{rem}\label{rem L_a^2}
For any $x\in A$, Proposition \ref{prop 2-gen}((i)\&(ii)) shows that $L_a^2(x) \in \ff L_a(x)+\ff a,$
and  $R_a^2(x) \in \ff R_a(x)+\ff a.$

We recall the 2-sided decomposition $x = \ga_xa + x_{0,0} +
x_{\gl,0} + x_{0,\gd}+ x_{\gl,\gd}$ of~equation~\eqref{eq
decomposition}.
\end{rem}

\begin{prop}\label{prop V_a(x)}
For $x\in A,$ define the vector space
\[
 V_a(x) = \ff a\oplus\ff x\oplus\ff ax\oplus \ff xa\oplus \ff axa.
\]
 Then
\[
V_a(x) = \ff a
\oplus  \ff x_{0,0} \oplus \ff  x_{\gl,0} \oplus \ff x_{0,\gd}\oplus
\ff x_{\gl,\gd}.
\]
\end{prop}
\begin{proof}
Applying a special case of Proposition~\ref{prop decomposition} both on the
left and right,
\[
\textstyle{\sum _{i,j=0}^{2} \ff L_a^iR_a^jx =\ff a
\oplus  \ff x_{0,0} \oplus \ff  x_{\gl,0} \oplus \ff
x_{0,\gd}\oplus \ff x_{\gl,\gd}.}
\]
We conclude with
Remark~\ref{rem L_a^2}, which lets us replace $L_a^2(y)$  and $R_a^2(y)$
by $L_a(y)$, $R_a(y),$ and $\ff a,$ for any $y.$
\end{proof}

As a useful corollary we get

\begin{prop}\label{prop V_a(b)}
$A=\ff a+\ff b_{0,0}+\ff b_{\gl,0}+\ff b_{0,\gd}+\ff b_{\gl,\gd}.$
\end{prop}
\begin{proof}
By Theorem \ref{thm 2-gen},
\[
A=\ff a + \ff b + \ff ab + \ff ba +\ff aba,
\]
so the proposition follows from Proposition \ref{prop V_a(x)}, because $A=V_a(b).$
\end{proof}

\begin{lemma}\label{lem ab=0}
If $ab=0,$ then $ba=0,$ and $A=\ff a+\ff b.$
\end{lemma}
\begin{proof}
Since $ab = 0,$ we get that $\ga_b a+\gl b_{\gl,0}+\gl b_{\gl,\gd}=0.$
It follows that $\ga_b=b_{\gl,0}=b_{\gl,\gd}=0.$  Hence $b=b_{0,0}+b_{0,\gd}.$
By Proposition \ref{prop V_a(b)}, $\dim(A)\le 3.$

If $b_{0,\gd}=0,$ then $ba=0,$ and we are done.  So we may assume that $b_{0,\gd}\ne 0.$
Now $0=b(ab)=(ba)b=\gd b_{0,\gd}b.$  Hence $b_{0,\gd}b=0.$  Thus
$A_0(R_b)=\ff a+\ff b_{0,\gd}$ is $2$-dimensional. Since $ab=0,$ we
see that $a=a_{0,0}+a_{\gl',0}.$  If $a_{\gl',0}=0,$ then $ba=0.$
If $a_{\gl',0}\ne 0,$ then $A_{\gl',0}(b)\ne 0,$
and we get that $\dim(A)\ge 4,$ a contradiction.
\end{proof}

\begin{lemma}\label{ab in Fa+Fb}
If $ab\in \ff a+\ff b,$ then $A=\ff a+\ff b$ has dimension $2.$
\end{lemma}
\begin{proof}
Write $ab=\ga a+\gb b.$  If $\gb =0,$ then, by Lemma \ref{lem
1.660}, $ab=0,$ so we are done by Lemma \ref{lem ab=0}.  Hence
$\gb\ne 0.$ We then see that
\[
\ga_b a+\gl(b_{\gl,0}+b_{\gl,\gd})=ab=\ga a+\gb (\ga_ab+b_{0,0}+b_{0,\gd}+b_{\gl,0}+b_{\gl,\gd}),
\]
so $b_{0,0}=0=b_{0,\gd}.$  Thus
$b=\ga_b a+b_{\gl,0}+b_{\gl,\gd}=\ga_b a+b_{\gl}.$  By
Proposition~\ref{prop V_a(b)}, $\dim(A)=3,$ and $A$ is spanned by $a, b_{\gl,0}, b_{\gl,\gd}.$
In particular, $A_0(L_a)=0.$  Note that $b_{\gl}^2\in \ff a+A_0(L_a),$
so $b_{\gl}^2\in\ff a.$  Now $ab=\ga_b a+\gl b_{\gl},$ hence $b_{\gl}\in\ff a+\ff b.$
Write $b_{\gl}=\ga_1 a+\gb_1 b,$ then, as above, using Lemma \ref{lem 1.660}, $\ga_1\ne 0\ne \gb_1.$
But now we see that
\[
\ff a \ni b_{\gl}^2=\ga_1^2 a+\gb_1^2 b+\ga_1\gb_1 ab+\ga_1\gb_1 ba.
\]
This shows that $ba\in\ff a+\ff b,$ so $A=\ff a+\ff b.$
\end{proof}

\begin{thm}\label{thm dim3}
Assume $\dim(A)=3.$  Then $a$ and $b$ are of Jordan type.
\end{thm}
\begin{proof}
By Lemma \ref{ab in Fa+Fb}, $a, b,$ and $ ab$ are independent; hence
$\gs\notin\ff a+\ff b.$ But $\gs=\gc a-\gl (b_{0,\gd}+b_{0,0}),$ so
at least one of $b_{0,0}, b_{0,\gd}$ is non-zero. Similarly, by
Lemma \ref{ab in Fa+Fb}, $a,b,ba$ are independent, so $\gs'\notin
\ff a+\ff b,$ and at least one of $b_{0,0}, b_{\gl,0}$ is non-zero.

Assume first that $b_{0,0}=0.$  Then $b_{\gl,0}\ne 0\ne b_{0,\gd}.$  Since $\dim(A)=3,$ Proposition \ref{prop V_a(b)} implies that $b_{\gl,\gd}=0.$
Thus $b=\ga_ba+b_{\gl,0}+b_{0,\gd}.$  Suppose $\ga_b=0.$  Then
\[
b=b^2=b_{\gl,0}^2+b_{0,\gd}^2+b_{\gl,0}b_{0,\gd}+b_{0,\gd}b_{\gl,0}.
\]
Since $b_{\gl,0}b_{0,\gd}+b_{0,\gd}b_{\gl,0}\in A_{\gl,\gd}=0,$ and $b_{\gl,0}^2+b_{0,\gd}^2\in \ff a+A_{0,0},$
we see that $b\in \ff a +A_{0,0},$ so $ab\in\ff a.$ By Lemma \ref{lem 1.660}, $ab=0,$ a contradiction.

Hence $\ga_b\ne 0.$  But then
\[
(ba)b=(\ga_ba+\gd b_{0,\gd})(\ga_ba+b_{\gl,0}+b_{0,\gd})=\ga_b^2a+\ga_b\gl b_{\gl,0} +\ga_b\gd^2b_{0,\gd}+\gd  b_{0,\gd}b_{\gl,0}+\gd b_{0,\gd}^2
\]
and
\begin{align*}
b(ab)&=(\ga_ba+b_{\gl,0}+b_{0,\gd})(\ga_b a+\gl b_{\gl,0})\\
&=\ga_b^2a+\ga_b\gd b_{0,\gd}+\ga_b\gl^2b_{\gl,0}+\gl b_{\gl,0}^2+\gl b_{0,\gd}b_{\gl,0}.
\end{align*}
Since $(ba)b=b(ab),$ this implies $\ga_b\gl b_{\gl,0}=\ga_b\gl^2b_{\gl,0},$ so $b_{\gl,0}=0,$ a contradiction.

Hence $b_{0,0}\ne 0.$  Suppose $b_{\gl,0}\ne 0.$  Then $b=\ga_b a+b_{0,0}+b_{\gl,0}.$  But then
$ba=\ga_b a,$ and we saw that this implies $ba=0,$ a contradiction.  Thus $b_{\gl,0}=0.$
Symmetrically $b_{0,\gd}=0,$ so $a$ is of Jordan type.  By symmetry so is~$b.$
\end{proof}

\section{A few observations and some questions.}
In this section $A$ is generated by a set of axes $X,$
where each $x\in X$ is primitive of type $(\gl_x,\gd_x).$
Recall the notation $\widebar{X}$ from \ref{not G(X)}.

\begin{lemma}\label{lem 00}
Let $a\in X.$
\begin{enumerate}\eroman
\item
For any $z\in A,$ there is a finite subset $Y\subseteq\widebar{X}$ such that
$z_{0,0}\in \sum_{x\in Y} \ff x_{0,0}.$ Here $z_{0,0}, x_{0,0}\in A_{0,0}(a).$

\item
To determine whether  $(A_{0,0}(a))^2 \subseteq A_{0,0}(a),$  it  is
enough to check that $(B_{0,0}(a))^2 \subseteq B_{0,0}(a),$ for any subalgebra
$B$ of $A$ generated by $a,x,x',$ with $x,x'\in \widebar{X}.$
\end{enumerate}
\end{lemma}
\begin{proof}
(i) This is obvious. Just write $z=\sum_{x\in Y}\gc_xx,$ for some
finite subset $Y\subseteq X.$ Then $z_{0,0}=\sum_{x\in
Y}\gc_xx_{0,0}.$
\medskip

\noindent
(ii)
Let $y',y\in A_{0,0}.$ Using (i) write
\[
y=\sum_{x \in Y}\gc_x x_{0,0},\qquad y'=\sum_{x' \in Y'}\gc'_{x'} x'_{0,0},
\]
with $Y, Y',$ finite subsets of $X.$  Then
\[
\textstyle{yy'\in \sum_{x\in Y, x'\in Y'}\ff x_{0,0}x'_{0,0}.}
\]
The assertion of (ii) is now obvious.
\end{proof}

\begin{lemma}\label{lem jordan} $ $
Suppose that whenever $B\subseteq A$ is generated by $2$
primitive axes of type $(\gl,\gd)$ and $(\gl',\gd'),$
then $\dim(B)\le 3.$  Then
all axes in $A$ are of Jordan type.
\end{lemma}
 \begin{proof}
This follows immediately from Theorem \ref{thm dim3} and Theorem \cite[Theorem A]{RoSe}.
\end{proof}

\subsection*{Some questions}\hfill
\medskip

\noindent
\begin{enumerate}
\item Suppose  $A$ is generated by axes $a$ and $b$ of respective types $(\gl,\gd)$ and  $(\gl',\gd')$.
Is it true that $\dim(A)\le 3 ?$  (see Lemma \ref{lem jordan}). We have recently shown that $\dim(A)\ne 4,$
so the only remaining case is whether one can have $\dim(A) = 5.$

\item
Suppose that each axis $x\in X$ is of (Jordan) type $(\gl_x,\gd_x).$
Is it true that $(A_{0,0}(x))^2 \subseteq A_{0,0}(x)?$
(By Lemma \ref{lem 00} it is enough to
consider triples of axes).

\item
Suppose $|X|=3.$ Does it follow that $\dim(A)$  is finite?
\end{enumerate}

\end{document}